\documentclass[11pt]{amsart}
\usepackage[height=615 pt,width= 360 pt]{geometry}
\usepackage{color}

\usepackage{amsthm,mathtools}
\usepackage{enumerate}

\newtheorem{theorem}{Theorem}

\newtheorem{lemma}[theorem]{Lemma}
\newtheorem{proposition}[theorem]{Proposition}
\newtheorem{definition}[theorem]{Definition}
\newtheorem{remark}[theorem]{Remark}

\usepackage{hyperref}

\title[Finsler broken scattering relation]{A foliated and reversible Finsler manifold is determined by its broken scattering relation}


\author[M. V. de Hoop]{Maarten V. de Hoop}
\address{M. V. de Hoop: Computational and
    Applied Mathematics \& Earth Science, Rice University, Houston,
    TX 77005, USA (\tt{mdehoop@rice.edu}).}

\author[J. Ilmavirta]{Joonas Ilmavirta}
\address{J. Ilmavirta: Unit of Computing Sciences, Tampere University
\linebreak
Tampere, FI-33014, Finland 
   (\tt{joonas.ilmavirta@jyu.fi})}

\author[M. Lassas]{Matti Lassas}
\address{M. Lassas: Department of Mathematics and Statistics, University of
\linebreak
 Helsinki, Helsinki, FI-00014, Finland   (\tt{matti.lassas@helsinki.fi})}
  
\author[T. Saksala]{Teemu Saksala}
\address{T. Saksala (corresponding author): Department of Mathematics,
  \linebreak
  North Carolina State University, Raleigh, NC 27607, USA
  \linebreak
  (\tt{tssaksal@ncsu.edu})}
  
\newcommand{\N}{\mathbb N}

\newcommand{\R}{\mathbb R}

\newcommand{\abs}[1]{\left\lvert #1 \right\rvert}
\newcommand{\der}{\mathrm d}

\newcommand{\ip}[2]{\left\langle#1,#2\right\rangle}
\newcommand{\sisus}{\operatorname{int}}
\newcommand{\id}{\operatorname{id}}
\newcommand{\eps}{\varepsilon}
\newcommand{\p}{\partial}

\newcommand{\inwb}[1]{\partial_{\mathrm{in}}SM_{#1}}
\newcommand{\inwbc}[1]{\partial_{\mathrm{in}}S{#1}}
\newcommand{\outwb}{\partial_{\mathrm{out}}SM}

\newcommand{\diffeo}{\xi}
\newcommand{\Diffeo}{\Xi}
\newcommand{\ddiffeo}{\eta}
\newcommand{\dm}{\phi}

\newcommand{\cut}{\tau_{\mathrm{cut}}}

\newcommand{\bou}{\tau_{\p M}}
\newcommand{\exit}{\tau_{\mathrm{exit}}}

\mathtoolsset{showonlyrefs}

\begin{document}

\begin{abstract}
The broken scattering relation consists of the total lengths of broken geodesics that start from the boundary, change direction once inside the manifold, and propagate to the boundary. We show that if two reversible Finsler manifolds satisfying a convex foliation condition have the same broken scattering relation, then they are isometric. This implies that some anisotropic material parameters of the Earth can be in principle reconstructed from single scattering measurements at the surface.
\end{abstract}

\maketitle
%
%
%
\section{Introduction}




The broken scattering relation of a Finsler manifold with boundary describes all the scenarios where a geodesic starts inward at the boundary, changes direction at some point in the interior, and makes its way back to the boundary.  
The objective in this paper is to reconstruct the Finsler manifold from such data, provided that some assumptions are met.
We have two key assumptions. One is that the Finsler geometry is reversible, meaning that the Minkowski norm on each tangent space is symmetric ($\abs{-v}=\abs{v}$) or, equivalently, that the reverse of a geodesic is a geodesic.
We also assume that the manifold has a strictly convex foliation with a family of smooth hyper surfaces. 
In this paper we will also describe an important relationship between Finsler geometry and elastic waves. 


\subsection*{Statement of the result}
Let $M$ be a smooth, compact manifold of dimension $3$ or higher, with smooth boundary $\p M$. We use the notation $TM$ for the tangent bundle of $M$. Recall that a Finsler metric $F\colon TM \to \R$ is a continuous positive function such that on each fiber $T_xM$ of $TM$ the function $F$ is a Minkowski norm meaning:
\begin{itemize}
    \item $F$ is smooth outside zero section.
    \item $F$ is positively homogeneous of order one, that is $F(x,av)=aF(x,v)$ for any $a>0$, $x \in M$ and $v \in T_xM$
    \item $F$ is convex in the sense that the local Riemannian metric
    \begin{equation}
    g_{ij}(x,v):=\frac{1}{2}\frac{\p^2}{\p v_i\p v_j}\left[F^2\right](x,v), \quad i,j \in \{1,\ldots, \dim M\}
    \end{equation}
    is positive definite for any $(x,v)\in TM$, $v\neq 0$. 
\end{itemize}
If $F$ is a Finsler metric we call the pair $(M,F)$ a Finsler manifold. If the Finsler metric is symmetric with respect to directional variable $v$ in the sense of $F(x,v)=F(x,-v)$ we call $F$ \textit{reversible} Finsler metric.

\begin{definition}
\label{def:foliation}
A Finsler manifold with boundary is said to have a strictly convex foliation if there is a smooth function $f\colon M\to\R$ so that
\begin{enumerate}[(i)]
\item $f^{-1}\{0\}=\partial M$,
$f^{-1}(0,S]=\sisus(M)$,
$f^{-1}(S)$ has empty interior
\item for each $s\in[0,S)$ the set $\Sigma_s\coloneqq f^{-1}(s)$ is a strictly convex smooth surface in the sense that $\der f\neq0$ and any geodesic $\gamma$, having initial conditions in $T\Sigma_s$, satisfies $\partial_t^2f(\gamma(t))|_{t=0}<0$.
\end{enumerate}
\end{definition}

In a similar fashion, we may define a strictly convex surface without a foliation.
For any surface $\Sigma$ there is a real-valued function $f$ defined in its neighborhood so that $f=0$ and $\der f\neq0$ on $\Sigma$.
The surface is called strictly convex if any geodesic $\gamma(t)$ tangent to $\Sigma$ at $t=0$ satisfies $\partial_t^2f(\gamma(t))|_{t=0}<0$.

In particularly we say that a reversible Finsler manifold $(M,F)$ has a strictly convex boundary $\p M$ if 
\[
f(x)=d(x, \p M), \quad x \in U,
\]  
satisfies the former conditions near the boundary. In this case any pair of points in $M$ can be connected by a distance minimizing geodesic \cite{bartolo2011convex}.
\color{black}

A manifold with a strictly convex foliation enjoys a certain weak non-trapping condition given in Lemma~\ref{lma:nontrap}.
If there is a trapped geodesic, it has to be asymptotically in the set $f^{-1}(S)$ where the foliating function~$f$ attains its maximum value, and it depends on the geometry of this set whether this is at all possible.

Next we describe what it means for two manifolds to have the same boundary data.
The Finsler function $F\colon TM\to\R$ describes the geometry, and boundary information requires knowledge of $F|_{\partial TM}$.
If we know the geometry of $\partial M$, we can deduce $F|_{T\partial M}$.
However, $T\partial M\subset\partial TM$ is a subbundle with fiberwise codimension one.
In Riemannian geometry the knowledge of the metric on $T\partial M$ determines uniquely the metric on $\partial TM$ in boundary normal coordinates.
In Finsler geometry it does not, because the geometry does not factorize as a product of tangential and normal directions of the boundary.
%
Also a diffeomorphism $\diffeo\colon\partial M_1\to\partial M_2$ only induces a diffeomorphism $\der\diffeo\colon T\partial M_1\to T\partial M_2$, and that is not enough for our purpose. Thus we introduces the following compatibility condition.

\begin{definition}
\label{def:compatible-diffeomorphism}
Let $M_1$ and $M_2$ be two smooth manifolds with boundary.
We say that a diffeomorphism $\Diffeo\colon\partial TM_1\to\partial TM_2$ is compatible with a diffeomorphism $\diffeo\colon\partial M_1\to\partial M_2$ if $\Diffeo$ is a linear isomorphism on every fiber and satisfies $\Diffeo(T\partial M_1)=T\partial M_2$ and $\Diffeo|_{T\partial M_1}=\der\diffeo$.
\end{definition}

Before showing that the broken scattering relation determines a Finsler manifold, we define what the relation is.
To this end, let $\phi_t\colon SM\to SM$ 
be the geodesic flow.
The natural projection of the tangent bundle will be denoted by $\pi\colon TM\to M$.
We will write elements of the tangent bundle either as sole vectors $v$ (which is in $T_xM$) or as pairs $(x,v)$, whichever is more convenient.

The boundary of the sphere bundle is
\begin{equation}
\partial SM
=
\{v\in SM;\pi(v)\in\partial M\}.
\end{equation}
We identify the inward-pointing part of this boundary,
\begin{equation}
\inwb{}
=
\{v\in\partial SM;\ip{v}{\nu}_\nu>0\},
\end{equation}
where $\nu$ is the inward pointing normal vector field and 
$
\ip{v}{\nu}_\nu:=g_{ij}(\nu)\nu^iv^j.
$
Similarly, we define the outward-pointing part
\begin{equation}
\outwb
=
\{v\in\partial SM;\ip{v}{\nu}_\nu<0\}.
\end{equation}

\begin{definition}
\label{def:R}
Let $(M,F)$ be a Finsler manifold with boundary.
For each $t>0$ we define a relation $R_t$ on $\inwb{}$ so that $v R_t w$ if and only if there exist two numbers $t_1,t_2>0$ for which $t_1+t_2=t$ and $\pi(\phi_{t_1}(v))=\pi(\phi_{t_2}(w))$.
We call this relation the broken scattering relation.
\end{definition}

Our main result is the following theorem stating that the  broken scattering relation, that is, the lengths of the broken geodesics, determine uniquely the isometry type of a Finsler manifold.

\begin{theorem}
\label{thm:main}
Let $(M_i,F_i), \: i \in \{1,2\}$ be two compact Finsler manifolds of dimension larger or equal to $3$, with boundary.
We assume the following:
\begin{enumerate}[(i)]
\item Both Finsler functions $F_1$ and $F_2$ are reversible.
\item The manifolds $(M_i,F_i), \:  i \in \{1,2\}$ have strictly convex foliations in the sense of definition~\ref{def:foliation}.
\item There are diffeomorphisms $\diffeo\colon\partial M_1\to\partial M_2$ and $\Diffeo\colon\partial TM_1\to\partial TM_2$ that are compatible in the sense of definition~\ref{def:compatible-diffeomorphism}.
\label{cond_0}
\item $F_1=F_2\circ\Diffeo$ on $\partial TM_1$. 
\label{cond_1}
\item For any two vectors $v,w\in\inwb{1}$ and $t>0$ we have $vR_t^{(1)}w$ if and only if $\Diffeo(v)R_t^{(2)}\Diffeo(w)$, where $R^{(i)}_t$ is the broken scattering relation of definition~\ref{def:R} on $(M_i,F_i)$.
\label{cond_2}
\end{enumerate}
Then there is a diffeomorphism $\dm\colon M_1\to M_2$ that is an isometry in the sense of $F_1=F_2\circ \der \dm$, which satisfies $\dm|_{\partial M_1}=\diffeo$, and $\der\dm|_{\partial TM_1}=\Diffeo$.
\end{theorem}

\begin{remark}
We note that to prove theorem \ref{thm:main} we will actually only need that $(M_1,F_1)$ has strictly convex foliation and $(M_2,F_2)$ has strictly convex boundary.

The requirement for the dimension $n\in \N$ of $M_i$ being larger than $2$ is a technical one. This is however a vital component of the proof of lemma \ref{lma:focusin_families_focus}. In the proof we need to assume that the dimension of a trace of geodesic, that is one, is strictly smaller than the dimension of the boundary $\p M_i$, which is $n-1$.   
\end{remark}

There are several well known examples of compact Riemannian manifolds satisfying a convex foliation condition of definition \ref{def:foliation}. One particularly interesting example to us is the closed unit ball $\bar B\subset \R^n$ with a radial metric $g(x)=c^{-2}(r)e$, where $r=|x| \in [0,1]$.  Here $c$ is a strictly positive smooth function on $[0,1]$ satisfying the Herglotz condition
\[
\frac{\der}{\der r}\left(\frac{r}{c(r)}\right)>0
\]
and $c'(0)=0$.
The condition does not forbid the existence of conjugate points but there are no trapped geodesics \cite{herglotz1905elastizitaet}.

Consider the smooth function $f(x)=|x|^2-1$ on $\bar B$.  It is well known that the radial Riemannian metric
$g(x)=c^{-2}(|x|)e$ satisfies the Herglotz condition if and only if the function $f$ satisfies the condition of definition \ref{def:foliation}.
Due to convexity, in terms of definition \ref{def:foliation}, being an open property we have that the function $f$ is a strictly convex foliation of a Finsler manifold $(\bar B,F)$ for any Finsler function $F$ on $\bar B$ that is close enough to $g$ in the $C^2$-topology.

Rotation symmetric Riemannian manifolds satisfying the Herglotz condition can have conjugate points.
Therefore some small Finslerian variations of such manifolds, at least those that preserve rotation symmetry, also have conjugate points.
Therefore our conditions do allow for conjugate points on non-Riemannian manifolds.

\subsection{Comparision of inverse problems for Riemannian and Finsler manifolds and the Riemannian counterpart of theorem \ref{thm:main}}
 Several inverse problems studied in literature have a different nature in Riemannian and Finslerian settings. 
For example, recall that a Riemannian surface is simple if it is 
 simply connected, its geodesic have no conjugate points and $\partial M$ is strictly convex. For the travel time inverse problem, discussed in detail below, it is shown in  \cite{pestov2005two} that
 the isometry type of a simple Riemannian manifold 
 is uniquely determined by the distances $d(x,y)$ of the boundary points $x,y\in \p M$, but there are counterexamples for this problem for simple Finsler manifolds~\cite{ivanov2013local}. Another example of the different nature of Riemannian and Finslerian settings is that the blow-up maps have been used to construct non-smooth counterexamples \cite{GLU,GKLU} for inverse problems for the Laplace-Beltrami operator, corresponding to a Riemannian metric. However, these maps can not be used for elastic equations, corresponding to a Finsler structure, as elastic media has to satisfy physical symmetries \cite{elasticcloaking}.

The Riemannian counterpart of Theorem \ref{thm:main}, presented in \cite{kurylev2010rigidity}, needs less assumptions. This is due to a certain rigidity of Riemannian geometry.
Heuristically, in Finsler geometry the metric in different directions at the same point are independent. This implies that in order to reconstruct a Finsler manifold in full, one needs access to the entire tangent bundle. To prove theorem \ref{thm:main} the first step is to show that the broken scattering relation determines the collection of distance functions at the boundary (see proposition \ref{prop:R-to-bdf}). However these distances only give an access for a certain open set of directions. 

The important feature of a Finsler function arising from a Riemannian metric tensor is that it is real analytic on every punctured tangent space.
If one makes this additional assumption, then the distance functions at the boundary determine the geometry at all points in all directions as proven in \cite{de2019inverse}. By lemma~\ref{lma:bdf-result} the distance function determines the Finsler function in directions where the geodesic is minimal to its endpoint on the boundary.
We call this the good subset of the tangent bundle.
These good directions exist on every tangent space, which is why analyticity gives access to the entire bundle.

The main result of \cite{de2019inverse} also asserts that the collection of the distance functions does not determine the Finsler function outside the good set of the tangent bundle.
\color{black}
Therefore on a general Finsler manifold one needs to study also long geodesics, and we have to go beyond the scope of \cite{kurylev2010rigidity}.
Reversibility and strict convexity give access to all directions near the boundary.
To go further into the manifold, we use the foliation: every point in the interior is close to the boundary when one goes deep enough in the foliation. 
%
%
\subsection{Elasticity and Finsler geometry}
\label{ssec:elasticity}




One can define elastic geometry in terms of distance:
The distance between two points can be declared to be the shortest time it takes for the support of a solution to the elastic wave equation, originally supported at one point, to reach the other point.

A more tractable description is obtained by studying the propagation of singularities of the elastic wave equation. 
There is a concrete way to pass from the stiffness tensor to a Minkowski norm on each tangent space, and this is described in~\cite{de2019inverse}.  
Microlocal analysis indicates \cite{duistermaat1996fourier, greenleaf1993recovering} that singularities follow the geodesic flow on the cosphere bundle of a Finsler metric, and cospheres are known as slowness surfaces in the physical literature. Elastic waves have three different polarizations in three spatial dimensions. The singularities used in the derivation of the Finsler geometry correspond to the fastest polarization known as quasi-pressure or qP.

In seismology the commonly used Preliminary Reference Earth Model~\cite{dziewonski1981preliminary} is spherically symmetric, and satisfies a foliation condition, the Herglotz condition, to great accuracy.
%
%
%
A typical inverse problem in elasticity is to reconstruct some elastic properties from boundary data \cite{bal2015reconstruction, bao2018direct, barcelo2018uniqueness, beretta2014uniqueness, hu2014recovering, liu2017decoupling, mazzucato2007uniqueness, nakamura2003global}.
Given this geometric point of view, the task is to find the Finsler geometry corresponding to the parameters.  

It was shown in \cite{de2019inverse} that Finsler metrics arising from elasticity are always reversible
and fiberwise real-analytic. The following result is an adaptation of theorem \ref{thm:main} in the elastic setting:
\begin{theorem}
\label{thm:elastic}
Let $(M_i,F_i), \: i \in \{1,2\}$ be two compact Finsler manifolds of dimension larger or equal to $3$, with strictly convex boundaries.  
We assume the following:
\begin{enumerate}[(i)]
\item Both Finsler functions $F_1$ and $F_2$ are reversible.
\item The manifolds $(M_i,F_i), \:  i \in \{1,2\}$ are fiberwise real analytic. That is for every $x \in M_i$ the function
\[
F_i(x,\cdot) \colon T_xM_i \to \R
\]
is real analytic.
\item There are diffeomorphisms $\diffeo\colon\partial M_1\to\partial M_2$ and $\Diffeo\colon\partial TM_1\to\partial TM_2$ that are compatible in the sense of definition~\ref{def:compatible-diffeomorphism}.
\item $F_1=F_2\circ\Diffeo$ on $\partial TM_1$. 
\item For any two vectors $v,w\in\inwb{1}$ and $t>0$ we have $vR_t^{(1)}w$ if and only if $\Diffeo(v)R_t^{(2)}\Diffeo(w)$, where $R^{(i)}_t$ is the broken scattering relation of definition~\ref{def:R} on $(M_i,F_i)$.
\end{enumerate}
Then there is a diffeomorphism $\dm\colon M_1\to M_2$ that is an isometry in the sense of $F_1=F_2\circ \der \dm$, which satisfies $\dm|_{\partial M_1}=\diffeo$, and $\der\dm|_{\partial TM_1}=\Diffeo$.
\end{theorem}

In the elastic setting no foliation condition is needed; it can be replaced with fiberwise analyticity.
But we do point out that with the foliation condition we have direct access to the Finsler function on all of the tangent bundle, whereas without the foliation condition we have to resort to analytic continuation which may be problematic for applications.
If our methods are used outside Hookean elasticity, the analyticity property 
and reversibility
may no longer be available.

\color{black}



\subsection{Related problems}

Let us consider a compact Riemannian manifold $(M,g)$ with boundary. The classical boundary rigidity problem is the following: We assume that we are given the distances $d(x, y)$, through $M$, of all boundary points $x,y \in \p M$. Can we determine the isometry type of the manifold $(M,g)$? Michel \cite{michel1981rigidite, michel1994restriction} observed that in the case of simple manifolds these distance functions also determine the values of the geodesic flow at the boundary, this is the scattering relation or lens relation:
\[
L = \{(v,w,t)\in \inwb{}\times \outwb{}\times \R: \phi_t(v)=w \hbox{ for some } t \leq 0\}.
\]
Thus $L$ carries the information when and where and in which direction a geodesic, sent from the boundary, hits the boundary again. 

The natural conjecture is that for simple manifolds the scattering relation determines the isometry type of the manifold (see \cite{burago2010boundary, croke1991rigidity, gromov1983filling, lassas2003semiglobal, michel1981rigidite, mukhometov1981problem, pestov2005two, stefanov2005boundary}). If the manifold is trapping one cannot determine the metric up to isometry if only the scattering relation is known  \cite{croke1992conjugacy}.   
 
\color{black}

On the other hand if a Riemannian manifold admits a suitable convex foliation condition, then a local version of the scattering rigidity problem, studied in \cite{stefanov2016boundary, stefanov2017local, uhlmann2016inverse}, imply the global boundary rigidity result. See for instance \cite[Section 2]{paternain2019geodesic} for a survey of different types of foliation conditions and geometric properties that imply their existence.  

Microlocal analysis connects singularities of solutions and solution operators to (hyperbolic) partial differential equations (PDEs) to geometry. The PDE related to our problem is the elastic wave equation. Its principal symbol is a matrix, which largest
eigenvalue is directly related to the Finsler metric \cite[Section 2]{de2019inverse}. Single scattering can be modeled by introducing, in the
elastic wave equation, a right-hand side representing a contrast source. The coefficients in this source, identified with the contrast in stiffness tensor, have a nonempty wavefront set. In the inverse problem studied in this paper, this wavefront set is assumed to be dense on the cosphere bundle associated with $M$. Thus in the case of very heterogeneous media with many scattering points inside the manifold one can obtain further information by looking at the propagation of singularities of waves going through the manifold. This is the broken scattering relation. 
Kurylev-Lassas-Uhlmann showed in \cite{kurylev2010rigidity} that this relation determines Riemannian manifold $(M,g)$, upto an isometry. They reduce the problem to the setup of \cite{kurylev1997multidimensional, Katchalov2001}, where it is shown that the collection of distance functions at the boundary determine the isometry class. In a sense they turn a data given by boundary sources, to one given by interior point sources. For inverse problems related to interior point sources and Riemannian wave equation see for instance \cite{maarten2019inverse, ivanov2019distance, LaSa,  lassas2018reconstruction, oksanen2011inverse}, and \cite{de2019inverse} for an interior source problem in Finsler geometry. 


 
\section{Outline of the proof}

We prove theorem~\ref{thm:main} in this section.
The proofs of the key lemmas are postponed to subsequent sections.
The rough plan of the proof is as follows:

\begin{itemize}
\item The first step is to verify that the broken scattering relation determines the boundary distance functions $\{d(x,\cdot)\colon \p M \to \R: x \in \hbox{int } M\}$. Then we show that the boundary distance function determines the topological and smooth structures. This is based on earlier work \cite{de2019inverse}. 
\item By the previous result the boundary distance function determines $F$ on the part of $TM$ from where the geodesic flow reaches the boundary in a sufficiently short time. This part is known as the ``good set'' $G\subset TM\setminus0$.
\item If a point $x\in M$ is sufficiently close to the strictly convex boundary, then more than half of the directions on $T_xM$ belong to $G$, so $F$ is determined there. 
\item By reversibility $F$ is determined on all of $T_xM$. This implies that we have found the Finsler geometry in all directions in a neighborhood of the boundary. Using the foliation, we write the neighborhood as  $f^{-1}([0,\eps])$ for some $\eps>0$.
\item We may then ask how far in the foliation the metric is uniquely determined. By the previous argument it holds at least for a little bit. If it only holds up to some $s<S$, then we may reiterate the argument on the smaller manifold $f^{-1}([s,S])$. To do so, we must propagate the data from the original boundary inward to the new one. Thus the uniqueness extends beyond the alleged limit $s$, proving uniqueness on the whole manifold.
\end{itemize}


We first give the lemmas and definitions needed to make the proof precise, then finish the proof, and finally complete the proof by proving the lemmas.

\subsection{Auxiliary results}

The boundary distance function of a point $x\in M$ is the function $r_x\colon\partial M\to\R$ defined by $r_x(y)=d(x,y)$. 
Notice that we assumed the Finsler metrics to be reversible, so it does not matter which direction we measure the distance in. 
In the cases where we are considering multiple manifolds, the boundary distance function of $x\in M_i$ is denoted by~$r_x^{(i)}$.

\begin{proposition}[Proven in section \ref{sec:From-BSR-to-BDF}]
\label{prop:R-to-bdf}
Let $(M_i,F_i)$, $i=1,2$, be two compact reversible Finsler manifolds, with strictly convex boundaries, whose broken scattering relations agree in the sense of conditions \textit{$\eqref{cond_0}$--\eqref{cond_2}} in theorem \ref{thm:main}.
Then the boundary distance functions agree in the sense that
\begin{equation}
\label{eq:BDF-agree}
\{
r_x^{(1)};x\in\sisus M_1
\}
=
\{
r_y^{(2)}\circ\diffeo;y\in\sisus M_2
\}.
\end{equation}
\end{proposition}

\begin{remark}
We note that in \eqref{eq:BDF-agree} we are given non-indexed sets. That is for a function $r_x\colon \p M\to \R$ we know whether it is included in the set \eqref{eq:BDF-agree}, but we do not know the indexing point $x \in M$. The map $\diffeo\colon\partial M_1\to\partial M_2$ is the same as in theorem \ref{thm:main}.
\end{remark}

We say that a direction $v\in TM\setminus0$ is minimizing if the maximal geodesic starting at $v$ reaches $\partial M$ in finite time and is a shortest curve joining its endpoints.
We denote the set of minimizing directions by $G$ (for ``good'') --- obviously with $G^{(i)}\subset TM_i\setminus0$.
The fibers are denoted by $G_x^{(i)}$.

\begin{lemma}[{\cite[theorem 1.3]{de2019inverse}}]
\label{lma:bdf-result}
Let $(M_1,F_1)$ and $(M_2,F_2)$ be two compact Finsler manifolds with boundary.
Suppose there is a diffeomorphism $\diffeo\colon\partial M_1\to\partial M_2$ so that \eqref{eq:BDF-agree} holds.
Then there is a diffeomorphism $\dm \colon M_1\to M_2$ so that $\dm|_{\partial M_1}=\diffeo$.
In addition, $F_1(v)=F_2(\der\dm(v))$ for all $v\in G^{(1)}$.
\end{lemma}

For a set $A\subset T_xM$ we denote by $-A$ the reflection and by $\sigma(A)=A\cup -A$ the symmetrization.
We also use the same notation on the whole bundle, applied fiberwise.

\begin{lemma}[Proven in section~\ref{sec:convex-half-sphere}]
\label{lma:convex-half-sphere}
If a Finsler manifold $(M,F)$ has a strictly convex boundary, then the set
\begin{equation}
U
=
\{x\in M;\sigma(G_x)=T_xM\setminus0\}
\end{equation}
is a neighborhood of the boundary $\p M \subset M$.
\end{lemma}

\begin{lemma}[Proven in section~\ref{sec:propagation}]
\label{lma:propagation}
Let $F_1$ and $F_2$ be two Finsler functions on a compact smooth manifold $M$ with boundary.
Suppose $(M,F_1)$ has a strictly convex foliation with a function $f\colon M\to[0,S]$.
Let $s\in(0,S)$ and denote $\widehat{M}=f^{-1}([s,S])$.
Assume the following:
\begin{itemize}
\item The two metrics coincide below $s$ in the sense that $F_1=F_2$ on $T(M\setminus\widehat{M})$.
\item For any two $v,w\in \p_{in}SM$ (these bundles coincide for the two metrics) and $t>0$ we have $vR_t^{(1)}w$ if and only if $vR_t^{(2)}w$, where $R_t^{(i)}$ is the broken scattering relation of $(M,F_i)$.
\end{itemize}
Then the scattering relations $\widehat{R}_t^{(i)}$ on $(\widehat{M},F_i)$ coincide in the sense that the assumptions of theorem~\ref{thm:main} are valid with $\diffeo$ and $\Diffeo$ being the identity maps.
Most importantly, $F_1=F_2$ on $\partial T\widehat{M}$ and for all $v,w\in\inwbc{\widehat{M}}$ we have $v\widehat{R}_t^{(1)}w$ if and only if $v\widehat{R}_t^{(2)}w$.
\end{lemma}

Finally we recall the second main result of \cite{de2019inverse}:
\begin{lemma}[{\cite[theorem 1.5]{de2019inverse}}]
\label{lma:bdf-elastic}
Let $(M_1,F_1)$ and $(M_2,F_2)$ be two compact Finsler manifolds with boundary.
Suppose there is a diffeomorphism $\diffeo\colon\partial M_1\to\partial M_2$ so that \eqref{eq:BDF-agree} holds.
If Finsler functions $F_1$ and $F_2$ are fiberwise real analytic, then there
exists a Finslerian isometry $\Psi\colon (M_1, F_1) \to (M_2,F_2)$ so that $\Psi| \p M_1 = \phi$.
\end{lemma}

\color{black}

\subsection{Proofs of the theorem{\color{red}s}}
Now we are ready to present the detailed proofs of theorems~\ref{thm:main} and ~\ref{thm:elastic} respectively.
We use the notations introduced in the previous subsection.

\begin{proof}[Proof of theorem~\ref{thm:main}]
It follows from proposition~\ref{prop:R-to-bdf} that the boundary distance functions of the two manifolds agree up to identifying the boundaries with~$\diffeo$.
By lemma~\ref{lma:bdf-result} there is a diffeomorphism
%
%
$\dm\colon M_1\to M_2$, and we use this diffeomorphism to identify the two manifolds as $M=M_1=M_2$.
This manifold inherits the foliation from $(M_1,F_1)$.

This smooth manifold $M$ has two Finsler functions $F_1$ and $F_2$, and the goal is to show that they are equal.
Lemma~\ref{lma:bdf-result} shows that $F_1=F_2$ in the good set $G:=G^{(1)}\subset TM\setminus0$.
By reversibility we conclude that $F_1=F_2$ in $\sigma(G)$.
Lemma~\ref{lma:convex-half-sphere} guarantees that $\sigma(G)$ contains the punctured tangent bundle of some neighborhood of $\partial M$.
Therefore there is $\eps>0$ so that $F_1=F_2$ in tangent bundle $T(f^{-1}([0,\eps]))$.


Now define
\begin{equation}
I
=
\{
h\in[0,S];
F_1(x,v)=F_2(x,v)
\text{ whenever $f(x)\leq h$ and $v\in T_xM\setminus0$}
\}.
\end{equation}
We showed that $[0,\eps)\subset I$, and it is clear that $I$ is an interval.
By continuity of the two Finsler functions, it is a closed interval and therefore $I=[0,s]$ for some $s\in(0,S]$.
If $s=S$, the two Finsler functions coincide on the whole tangent bundle and the proof is complete.

Suppose then that $s\in(0,S)$.
We are now in the setting of lemma~\ref{lma:propagation}.
The two metrics coincide in the strip $f^{-1}([0,s])$, so the data may be propagated through it.
We have now two metrics on the shrunk manifold $\widehat{M}=f^{-1}([s,S])$ and their broken scattering relations coincide by lemma~\ref{lma:propagation}.

Let us denote the good set of $(\widehat M,F_i)$ by $\widehat G^{(i)}$.
Repeating the argument obtained above, we find that there is a diffeomorphism
$\ddiffeo\colon\widehat M\to\widehat M$ so that $F_1=\ddiffeo^*F_2$ on $\widehat G^{(1)}$, which again has full fibers when the base point is in some neighborhood of the boundary.
It is straightforward to check that if $(x,v)\in G$ and $x\in\widehat M$, then $(x,v)\in\widehat G^{(1)}$.
We also note that for any $x \in \widehat M$ the intersection $T_x \widehat M \cap G^{int}$ is non-empty, since any small perturbation of an initial direction of the geodesic connecting $x$ to a closest boundary point in $\p M$, is contained in $G$. Take $(x,v)\in G^{int}$ and let $\gamma$ be the forward-maximal geodesic with respect to $F_1$ with such initial data. Due to \cite[Section 3.3]{de2019inverse} it follows that the lift of $\gamma$ is contained in $G^{int}$. Since the geodesic coefficients $G^i(w), \: i \in \{1,\ldots,n\}, \: w \in TM$ (see for instance \cite[formula (5.7)]{shen2001lectures} for the definition) of $F_1$ and $F_2$ agree on $G^{int}$ it follows that
$\gamma$ is also a geodesic with respect to $F_2$.

But as $F_1$ and $\ddiffeo^*F_2$ agree in a neighborhood of the lift of $\gamma$, we have that the $F_2$-geodesic $\gamma$ starts at $\ddiffeo(x)$.
Thus $\ddiffeo(x)=x$ and one could choose any starting point $x\in\widehat M$, and so $\ddiffeo=\id$.
Therefore $F_1(x,v)=F_2(x,v)$ for $x\in\widehat M$ close enough to $\partial\widehat M$ and all $v\in T_xM$.
This means that the two metrics coincide in $f^{-1}([0,s+\eps])$ for some $\eps>0$, contradicting the maximality of $s$.

Finally we verify that
$
\der \dm$ and $\Diffeo
$
coincide on $\p TM_1$.
First we note that definition \ref{def:compatible-diffeomorphism}, proposition \ref{prop:R-to-bdf} and lemma \ref{lma:bdf-result} imply
\begin{equation}
\label{eq:boundary_differential}
\der \dm|_{T \p M_1}=\der \diffeo= \Diffeo|_{T \p M_1}.
\end{equation}
We recall that  $\Diffeo$ and $\der \dm$ are linear on the fibers, and therefore to conclude the proof it suffices to show that $\Diffeo$ preserves the inward pointing  unit normal vector field to the boundary. By taking the directional differential $\der_v$ of condition \eqref{cond_1} of theorem \ref{thm:main}, and using equation \eqref{eq:boundary_differential}, we get
$
\langle\der_v(F_2)(\Diffeo \nu),\der \diffeo w\rangle=0, \: w \in T\p M_1.
$

Since $\diffeo\colon \p M_1 \to \p M_2$ is a diffemorphism we have shown that the Legendre transform of $\Diffeo \nu$ is co-normal to $T\p M_2$. This implies that $\Diffeo \nu $ is normal to the boundary $\p M_2$. Due to condition \eqref{cond_2} of theorem \ref{thm:main} it holds that $\Diffeo$ maps ${\inwb{}}_1$ onto ${\inwb{}}_2$. Therefore we have verified
$
\Diffeo \nu
$
is the inward pointing unit normal to $\p M_2$. The proof is complete.
\end{proof}

\begin{proof}[Proof of theorem~\ref{thm:elastic}]
Theorem~\ref{thm:elastic} is a direct consequence of proposition \ref{prop:R-to-bdf} and lemma \ref{lma:bdf-elastic}.
\end{proof}

\color{black}


\section{Short and long geodesics}
\label{sec:convex-half-sphere}

In this section we analyze short geodesics near the boundary to prove lemma~\ref{lma:convex-half-sphere}.
To do so, we first make an observation concerning the lengths of geodesics of arbitrary --- even infinite --- length.

For any Finsler manifold $(M,F)$ with boundary we define the exit time function $\exit\colon SM\to[0,\infty]$ so that $\exit(v)$ is the first time $t>0$ for which $\pi(\phi_t(v))\in\partial M$.
If there is no such time, the value is taken to be infinity.
If $v$ is based at a boundary point and points tangentially or outward, we set $\exit(v)=0$. 
\begin{lemma}
\label{lma:exit-continuous}
On any complete Finsler manifold $M$ with strictly convex boundary the exit time function $\exit\colon SM\to[0,\infty]$ is continuous.

Without assuming completeness, continuity holds where $\exit<\infty$, including a neighborhood of $\overline{\outwb}$.
\end{lemma}

\begin{proof}
We will check continuity separately at different kinds of points $v\in SM$, depending on the value $\exit(v)$.

If $\exit(v)=0$, then $v$ is based on the boundary and points tangentially or outwards.
(That is, $v\in\overline{\outwb}$.)
By the assumption of strict convexity, any geodesic starting near $v$ is short due to \cite[Section 8.1]{SUbook}, establishing continuity of $\exit$ at $v$.

If $\exit(v)\in(0,\infty)$, then the geodesic $\gamma_v$ starting at $v$ reaches $\partial M$ in finite time.
By strict convexity $\dot\gamma_v(\exit(v))$ is transverse to $\partial M$, and it follows from the implicit function theorem that $\exit$ is in fact smooth in a neighborhood of $v$.

If $\exit(v)=\infty$, then the geodesic $\gamma_v$ is trapped.
Continuity at $v$ can only fail if there is a sequence of vectors $v_k\in SM$ so that $v_k\to v$ as $k\to\infty$ and $\exit(v_k)\leq T$ for some $T\in(0,\infty)$.
There is some $r>0$ so that for all $k\in\N$ we have $d(\pi(v_k),\pi(v))<r$.
All the forward-maximal geodesics $\gamma_{v_k}$ are thus contained in the metric closed ball $K=\bar B(\pi(v),r+T)$.
By completeness of $M$ the set $K$ is compact. Up to extracting a subsequence, the endpoints $\gamma_{v_k}(\exit(v_k))\in\partial M\cap K$ converge to a point $z\in\partial M$ and $\exit(v_k)$ converges to $T'\leq T$.
By continuity of the geodesic flow, we have that $\phi_{T'}(v)=\lim_{k\to\infty}\phi_{\exit(v_k)}(v_k)\in S_zM$.
This means that $\gamma_v$ meets the boundary at time $T'<\infty$, which is a contradiction.

We conclude that $\exit$ is continuous at all points of $SM$.
Completeness was only needed to prove continuity where $\exit=\infty$, and that cannot happen near $\overline{\outwb}$ when $\partial M$ is strictly convex.
\end{proof}

\begin{proof}[Proof of lemma~\ref{lma:convex-half-sphere}]
If the manifold is not compact, we can restrict the analysis to a bounded neighborhood of a given boundary point and do everything on a compact submanifold.
On a compact Finsler manifold there is some $\eps>0$ so that any geodesic with length less than $\eps$ is minimizing.

Consider any point $p\in\partial M$, and let $\nu_p$ be the inward unit normal at $p$ and $\gamma_p$ the geodesic starting in the direction of $\nu_p$.
For any $t>0$ small enough so that the geodesic $\gamma_p$ up to time $t$ exists we define the map $P^p_t\colon S_pM\to S_{\gamma_p(t)}M$ by parallel transport.

Fix some $T>0$ small enough.
We define a function $Q\colon[0,T)\times\partial SM\to SM$ so that $Q(t,v)=P^{\pi(v)}_t(v)$.
This function is continuous, so by lemma~\ref{lma:exit-continuous} the composed function $\exit\circ Q\colon[0,T)\times\partial SM\to[0,\infty]$ is also continuous.
(Continuity will only be needed near directions where $\exit=0$.) 
This composed function vanishes on $\{0\}\times\overline{\outwb}$, so there is an open neighborhood $V\subset[0,T)\times SM$ of $\{0\}\times\outwb$ so that $\exit\circ Q|_V<\eps$.
Since all geodesics shorter than $\eps$ are minimizing, we have $Q(V)\subset G$.

Any boundary point $p\in\partial M$ has a neighborhood $W\subset\partial M$ and two numbers $a,h>0$ so that $[0,h)\times V_p\subset V$ with
\begin{equation}
V_p
\coloneqq
\{v\in SM;\pi(v)\in W\text{ and }\ip{v}{\nu}_\nu<a\}
.
\end{equation}
We note that $V_p$ contains larger portion of $S_pM$ than the Southern hemisphere. We have that $Q([0,h)\times V_p)\subset G$ and
\begin{equation}
Q([0,h)\times\sigma(V_p))
=
\sigma(Q([0,h)\times V_p))
\subset
\sigma(G).
\end{equation}
But we have
$
\sigma(V_p)
=
\{v\in SM;\pi(v)\in W\},
$
so that the set
$
U'
=
\{
\gamma_p(t);t\in[0,h)\text{ and }p\in W
\}
$
is a subset of the set $U$ defined in the claim of the lemma.
Now $U'$ is a neighborhood of the arbitrary boundary point $p$, so we have established that $U$ does indeed contain a neighborhood of the boundary.
\end{proof}

\section{Propagation of data through a layer}
\label{sec:propagation}

In order to prove lemma~\ref{lma:propagation} to propagate data from $\partial M$ to $\partial\widehat M$, we first observe that a foliated manifold enjoys a certain non-trapping property. 

\begin{lemma}
\label{lma:nontrap}
Let $(M,F)$ be a compact Finsler manifold which is foliated by a smooth function $f\colon M\to\R$ in the sense of definition~\ref{def:foliation}.
If a maximal geodesic $\gamma$ satisfies $\partial_t(f(\gamma(t))<0$ at $t=0$, then $\gamma$ reaches $\partial M$ in finite time.
\end{lemma}

\begin{proof}
Consider the function $h(t)=f(\gamma(t))$, defined over some maximal interval $[0,T]$ or $[0,\infty)$.
The goal is to show that $h$ obtains the value zero in finite time.
Given that $h'(0)<0$ and $h$ is smooth, this can only fail if one of the following happen:
\begin{enumerate}
\item[(i)]
\label{case1}
The derivative $h'(t)$ vanishes for some $t$.
\item[(ii)]
\label{case2}
The function has a limit: $h$ is strictly decreasing and $\lim_{t\to\infty} h(t)=H\in[0,\infty)$.
\end{enumerate}

Let us first exclude case~(i).
If $h'$ vanishes somewhere, there is a smallest $t>0$ for which $h'(t)=0$.
We have that $h'<0$ on $[0,t)$.
The geodesic $\gamma$ is tangent to the level set $f^{-1}(h(t))$, so by definition~\ref{def:foliation} we have $h''(t)<0$.
This implies that for some $\eps>0$ we have $h'(t-\eps)>0$, which is a contradiction.

Let us then move to case~(ii).
The curve $\gamma$ now approaches the surface $\Sigma\coloneqq f^{-1}(H)$.
By monotonicity there is an increasing sequence of times $t_k\to\infty$ so that $h'(t_k)\to 0$.
Upon extracting a subsequence, there exists a point $x\in \Sigma$ so that $\gamma(t_k)\to x$ and $\dot\gamma(t_k)$ converges to some $v\in S_xM$.
As $h'(t_k)\to 0$, it follows that $v$ is tangent to $\Sigma$.

Let $\tilde\gamma$ be the geodesic starting with initial conditions $(x,v)$.
We denote $\tilde h=f\circ\tilde\gamma$.
Since $v$ is tangent to $\Sigma$, we have $\tilde h'(0)=0$, so definition~\ref{def:foliation} implies $\tilde h''(0)<0$.
%
%
By smoothness of $f$ and the time additive property of geodesic flow we have
\begin{equation}
\lim_{k\to\infty}
h''(t_k+s)
= 
\tilde h''(s) 
\end{equation}
for any $s\in\R$.

Because $\tilde h''(0)<0$, there are $\eps>0$ and $\delta>0$ so that $h''(t_k+s)\leq-\eps$ for all $s\in[-\delta,\delta]$ when $k$ is large enough.
This together with $h'\leq0$ (which is due to monotonicity) and the fundamental theorem of calculus gives
\begin{equation}
h(t_k+\delta)
=
h(t_k-\delta)
+
2\delta h'(t_k-\delta)
+
\int_{-\delta}^\delta
\int_{-\delta}^s
h''(t_k+r)
\der r
\der s
\leq
h(t_k-\delta)
-2\eps\delta^2
\end{equation}
for large enough $k$.
As $-2\eps\delta^2$ is independent of $k$ and $h$ is monotonous, it follows that $h$ cannot converge.
This is a contradiction.

Therefore there is indeed some finite $t$ so that $h(t)=0$.
\end{proof}


\begin{proof}[Proof of lemma~\ref{lma:propagation}]
Equality of the Finsler functions on $\partial\widehat M$ to all directions on the bundle --- which amounts to $\diffeo$ and $\Diffeo$ being identities --- follows from the assumed equality of the Finsler functions on the strip $f^{-1}([0,s])$ and continuity.

For any $v\in SM$, we denote by $\gamma_v$ the unique geodesic with $\dot\gamma_v(0)=v$.
Take any two $v,w\in\inwbc{\widehat M}$ and $t>0$.

By lemma~\ref{lma:nontrap} there is $a>0$ so that the geodesic segment $\gamma_v|_{[-a,0]} \subset f^{-1}([0,s])$ connects a point on $\partial M$ to the point $\pi(v)\in\partial\widehat M$.
Because the two Finsler metrics agree in $f^{-1}([0,s])$, this same curve is a geodesic in both geometries.
By strict convexity this geodesic meets $\partial M$ transversely, so that $v'\coloneqq\dot\gamma_v(-a)$ belongs to $\inwb{}$.
Similarly, there are $b>0$ and $w'\in\inwb{}$ corresponding to $w$.

Because the two geodesic flows agree from $v'$ to $v$ and from $w'$ to $w$, we have that $v\widehat R_t^{(i)}w$ if and only if $v'R_{t+a+b}^{(i)}w'$ for both $i \in \{1,2\}$.
The broken scattering relations $R^{(i)}_{t+a+b}$ agree on $\partial M$, and so the broken scattering relations $\widehat R_t^{(i)}$ agree on $\partial\widehat M$.
%
\end{proof}

\section{From broken scattering relation to boundary distance functions}
\label{sec:From-BSR-to-BDF}

In this section we will prove proposition~\ref{prop:R-to-bdf}. We start with considering only one Finsler manifold $(M,F)$ with strictly convex boundary whose broken scattering relation is known.

\subsection{Critical distance functions}
\label{ssec:focus_fam}

We begin with studying classical critical distance functions. The first one is the cut distance function $\cut\colon S M\to (0,\infty]$ given by
\begin{equation}
\cut(x,v) \coloneqq
\sup
\{t>0; \gamma_{x,v}(t) \text{ exists and }  d(x,\gamma_{x,v}(t))=t\}.
\end{equation}
Since the boundary of the $M$ is strictly convex it holds that any distance minimizing curve is a geodesic. Therefore function $\cut$ is well defined and moreover it is continuous. 

Next we formulate two auxilliary lemmas related to $\cut$ function.

\begin{lemma}
\label{Le:self_int}
For any $(x,v) \in \inwb{}$  and 
$
s_2 > s_1 \geq 0,
$
which satisfy
\begin{equation}
\label{eq:self_intersection_of_gamma}
\gamma_{x,v}(s_1)=\gamma_{x,v}(s_2),
\quad \hbox{we have} \quad  s_1+s_2 > 2\cut(x,v).
\end{equation}
\end{lemma}

\begin{proof}
Since $F$ is reversible this claim can be proven like \cite[lemma 2.1]{kurylev2010rigidity}.
\end{proof}



\begin{lemma}
\label{Le:con_of_total_time}
Let $z\in \p M$, $t\in (0,\cut(z,\nu(z)))$ and $U \subset TM$ be a neighborhood of $\nu(z)$ that is diffeomorphic to some open set of $\R^{2n-1}\times [0,\infty)$. For any $\epsilon> 0$ we can choose $\delta=\delta(z,t,\epsilon)>0$ such that the following holds: If $v_i \in U, \: i\in \{1,2\}$ satisfy
\begin{equation}
v_1 R_{2t}v_2, \quad \hbox{ and }  \quad \|v_i-\nu(z)\|_e< \delta,
\end{equation}
then there exist $t_1,t_2>0$, so that  
\begin{equation}
\gamma_{v_1}(t_1)=\gamma_{v_2}(t_2),\quad    t_1+t_2 =  2t, \quad \hbox{ and } \quad |t_i-t| <\epsilon.
\end{equation}
Here $\|\cdot\|_e$ is the Euclidean norm on $U$.
\end{lemma}
\begin{proof}
The proof is analogous to the proof of \cite[lemma 2.2]{kurylev2010rigidity}.
\end{proof}


\medskip


The second critical distance functions is the boundary cut distance function
\begin{equation}
\bou(z):=\sup \{t>0:\: d(z,\gamma_{z,\nu(z)}(t))=t= d(\partial M,\gamma_{z,\nu(z)}(t))\}, \quad z \in \p M.
\end{equation}
%
%
%
These two critical distance functions satisfy the following:

\begin{lemma}[{\cite[lemma 3.8]{de2019inverse}}]
\label{Le:comp_of_crit_func}
For any $z\in \partial M$ it holds that 
\begin{equation}
\cut(z,\nu(z)) > \bou(z). 
\end{equation}
\end{lemma}

\subsection{Family of focusing directions}

Let $(M,F)$ be a compact Finsler manifold of dimension $3$ or higher with reversible Finsler function $F$ and strictly convex boundary. Let $x \in M^{int}$ and $v \in S_xM$ be a direction such that the corresponding geodesic $\gamma_v$ is the distance minimizer from $x$ to $z_x\in \partial M$ a closest boundary point to $x$. In the proof of lemma \ref{lma:exit-continuous} we showed that the exit time function $\exit$ is smooth in some neighborhood $V \subset S_x M$ of $v $, and moreover implicit function theorem implies that the map
\[
H\colon V \ni \eta \mapsto
\pi(\phi_{\exit(\eta)}(\eta))\in \p M
\]
is a diffeomorphism to some neighborhood $U\subset \p M$ of $z_x$. Using this identification we define a smooth inward pointing unit length vector field on $U$
\[
V(z):=-\phi_{\exit(\eta)}(\eta), \quad \eta:=H^{-1}(z).
\]
This vector field satisfies 
\begin{equation}
V(z_1)R_{T(z_1,z_2)}V(z_2), \hbox{ and } V(z_x)=\nu(z_x),
\end{equation}
where 
\begin{equation}
T(z_1,z_2):=\exit(H^{-1}(z_1))+\exit((H^{-1}(z_2)).
\end{equation}
Moreover, the function
\begin{equation}
t\colon U \to \R, \quad t(z):=\frac{1}{2}T(z,z)=\exit((H^{-1}(z))
\end{equation}
is smooth and its differential vanishes at $z_x$
which is a closest boundary point to $x$. Thus the geodesics given by initial conditions $(z,V(z)), \: z \in U$ focus at the common interior point $x$ at time $t(z)$.  

We change the point of view and set the following definition:
\begin{definition}
\label{def:focusing_family}
Let $z_0\in \partial M$ and $t_0 \in (0,\exit(\nu(z_0)))$. We say that a collection 
\begin{equation}
F(z_0,t_0):=\{U, V(\cdot),t(\cdot)\}
\end{equation}
where $U\subset \p M$ is a neighborhood of $z_0$, $V\colon U\to \partial_{in}SM$ is a smooth vector field and  $t\colon U \to (0,\infty)$ is a smooth  function, is called a family of focusing directions around $(z_0,t_0)$ if  
\begin{align}
\label{eq:focusing_family_1}
&V(z_1)R_{t(z_1)+t(z_2)}V(z_2), \: \hbox{ for all } z_1,z_2 \in U, 
\\
\label{eq:focusing_family_2}
&V(z_0)=\nu(z_0), \quad t(z_0)=t_0,  \quad  \hbox{ and }  \quad  \der t(z)\bigg|_{z=z_0}= 0.
\end{align}
\end{definition}

We note that the broken scattering relation determine all the families of focusing directions, but it is possible that not all of them focus in the sense of
\begin{equation}
    \label{eq:focus_real}
    \pi(\phi_{t(z)}(V(z)))=\pi(\phi_{t_0}(\nu(z_0))), \quad \hbox{ for all } z \in U.
\end{equation}
Next we give the following result that guarantees the actual focusing if the focusing time $t_0$ is small enough. 
\begin{lemma}
\label{lma:focusin_families_focus}
Let $z_0\in \partial M$, $t_0 \in (0,\exit(\nu(z_0)))$ and $F(z_0,t_0)$
\\
 $=\{U', V(\cdot),t(\cdot)\}$ be a family of focusing directions around $(z_0,t_0)$. If
\begin{equation}
t_0<  \cut(z_0,\nu(z_0))=:\tau_{M}(z_0),
\end{equation}
then there exists a neighborhood $U \subset U'$ of $z_0$ such that \eqref{eq:focus_real} holds true.
\end{lemma}

For the proof of the lemma we need the following auxiliary result.

\begin{lemma}
\label{Le:intersect_of_geo_and_foc_fam}
Let $z_0 \in \p M,\: t_0>0$ and the family of focusing directions $F(z_0,t_0)=\{U,V(\cdot),t(\cdot)\}$ around $(z_0,t_0)$ be as in lemma \ref{lma:focusin_families_focus}. Let $\gamma$ be some geodesic that intersects $\gamma_{V(z_0)}$ at $r_0$ transversely. If there exist functions $r,\rho\colon U \to \R$  for which hold
\begin{equation}
\rho(z_0)=0,\quad r(z_0)=r_0, \quad \gamma(\rho(z))=\gamma_{V(z)}(r(z)),
\end{equation}
and 
\begin{equation}
0\leq r(z)\leq r_1< \tau_{M}(z_0), \quad |\rho(z)|\leq \rho_1< \hbox{inj}(M).
\end{equation}
Then
$
t_0=r_0.
$
\end{lemma}

\begin{proof}
The proof consists of several steps.

\medskip
\textbf{(I)} Since $F(z_0,t_0)$ is a family of focusing directions the equation \eqref{eq:focusing_family_1} implies the existence of functions $s,\widehat{s}\colon U \to [0,\infty)$ that satisfy
\begin{equation}
\label{eq:func-s-hat-s}
\gamma_{V(z)}(s(z))=\gamma_{V(z_0)}(\widehat{s}(z)), \quad s(z)+\widehat{s}(z)=t(z)+t_0.
\end{equation}
Since vector field $V$ is continuous, equation \eqref{eq:focusing_family_2} and lemma \ref{Le:con_of_total_time} imply
\begin{equation}
s(z)\to t_0, \quad \widehat s(z) \to t_0, \quad \hbox{ as } z \to z_0.
\end{equation}
Therefore 
$
s(z_0)=\widehat{s}(z_0)=t_0,
$
and by an analogous proof to one given in \cite[lemma 2.8]{kurylev2010rigidity}
we show that the functions $s, \: \widehat{s}$ are smooth near $z_0$ and satisfy
\begin{equation}
\der s(z_0)=\der \widehat{s}(z_0)=0.
\end{equation}

\medskip

\textbf{(II)} 
Let $W\subset \p M$ be a neighborhood of $z_0$ where functions $s,\widehat{s}$ are smooth. We consider a map
\begin{equation}
E\colon W \to SM, \quad E(z):=-\phi_{s(z)}(V(z)).
\end{equation}
We begin with studying the differential
$
\der E|_{z_0}\colon T_{z_0}\p M \to T_{(x,\eta)}SM, \: (x,\eta):=E(z_0).
$
Recall that the tangent bundle $T(SM)$ has a canonical decomposition to  Horizontal $H(SM)$ and Vertical $V(SM)$ sub-bundles. We denote the projections from $T(SM)$ to these bundles by $P_H$ and $P_V$ respectively. 
We note first that
\begin{equation}
x(z):=(\pi \circ E)(z)=\gamma_{V(z)}(s(z))=\gamma_{\nu}(\widehat s(z)).
\end{equation}
Also a simple computation shows that after identifying $H_{(x,\eta)}(SM)$ to $T_xM$ we have
$
\der x|_{z_0}=(P_H\circ(\der E))|_{z_0}.
$
Therefore
\begin{equation}
(P_H\circ(\der E))|_{z_0}=\der\left[\gamma_{V(z)}(s(z))\right]\bigg |_{z=z_0}\! \! \!  =\der \left[\gamma_{\nu}(\widehat s(z))\right]\bigg |_{z=z_0} \! \!\!   =\dot{\gamma}_\nu(t_0)\otimes \der \widehat s(z_0)=0,
\end{equation}
implies that
\begin{equation}
\der E\bigg|_{z_0}\colon T_{z_0} \p M\to T_{(x,\eta)} SM =H_{(x,\eta)}(SM)\oplus V_{(x,\eta)}(SM), \quad \der E\bigg|_{z_0} v=(0,\Theta v),
\end{equation}
where 
$
\Theta \colon T_{z_0}\p M \to V_{(x,\eta)}(SM)
$
is a linear map. Recall that $V_{(x,\eta)}(SM)$ is isomorphic to $T_\eta(S_xM)$, which is $(n-1)$-dimensional as also $T_{z_0}\p M$. Therefore 
to show that $\Theta$ is a linear isomorphism it suffices to prove that it is injective. We define a map
\begin{equation}
G(z):=\exp_{x(z)}(s(z)E(z))=z, \quad \hbox{ for all } z \in W.
\end{equation}
Thus the differential $\der G$ at $x_0$ is an identity operator on $T_{z_0}\p M$. Since $s(z_0)=t_0$ is less than the cut distance $\cut(z_0,\nu(z_0))$ and $\der s(z)|_{z=z_0}=\der x(z)|_{z=z_0}=0$ we have
\begin{equation}
\der G\bigg|_{z=z_0}v=\der(\exp_{x})\bigg|_{t_0E(z_0)}t_0\der E|_{z=z_0}v=\der(\exp_{x})\bigg|_{t_0 E(z_0)}t_0\Theta v=v.
\end{equation}
This implies that $\Theta$ is an injection.

\medskip

\textbf{(III)}
Now we prove that the function $r$ given in the claim of this lemma is continuous at $z_0$. If this is not true there exist $\epsilon>0$ and a sequence $z_k\in \p M$ which converges to $z_0$ but for which hold
\begin{equation}
|r_k-r_0|>\epsilon, \quad r_k:=r(z_k).
\end{equation}
Since functions $r$ and $\rho$ are bounded we can without loss of generality assume that 
$
r_k\to r' <\tau_{M}(z_0), \: |r'-r_0| \geq \epsilon, \: \rho_k:=\rho(z_k)\to \rho'\in \R, \hbox{ such that } |\rho'|< \hbox{inj}(M).
$
Thus
\begin{equation}
\gamma(\rho')=\lim_{k \to \infty}\gamma(\rho_k)=\lim_{k \to \infty}\gamma_{V(z_k)}(r_k)=\gamma_{\nu}(r').
\end{equation}
Since $r_0,r' <\tau_{M}(z_0) $ we have that
$
x:=\gamma_{\nu}(r_0)=\gamma(0), \hbox{ and } x':=\gamma_{\nu}(r')=\gamma(\rho')
$
are two different points where $\gamma$ and $\gamma_\nu$ intersect. Since
$
r_0,r'< \tau_{M}(z_0),$ and 
$|\rho'|< \hbox{inj}(M),
$
there are two different distance minimizing geodesics connecting $x$ to $x'$, which is not possible.

\medskip
\textbf{(IV)} Let us then assume that $r_0< t_0.$ We study a map
\begin{equation}
\Phi\colon U \times \R \to M, \quad \Phi(z,\lambda)=\exp_z(\lambda V(z)),
\end{equation}
and show that it is a local diffeomorphism near $(z_0,r_0)$. Since $0<t_0-r_0< \tau_{M}(z_0)$ and the Finsler function $F$ is reversible, the map $\exp_{x_0}$ is a local diffeomorphism near $(t_0-r_0)\eta$, where
$
x_0=\gamma_{\nu}(t_0-r_0), \: \eta_0 =-\dot \gamma_{\nu}(t_0-r_0).
$
Thus
\begin{equation}
\der\exp_{x_0}\bigg|_{(t_0-r_0)\eta_0}\colon T_{(t_0-r_0)\eta}(T_{x_0}M) \to T_{z_0}M 
\end{equation}
is a linear isomorphism. Next we note that 
\begin{equation}
\Phi(z,\lambda)=\gamma_{E(z)}(s(z)-\lambda)=\exp_{x(z)}((s(z)-\lambda)E(z)),
\end{equation}
where $E$ and $s(\cdot)$ are the same maps as in parts \textbf{(I)} and \textbf{(II)} of the proof. Since the differentials of $x$ and $s$ vanish at $z_0$ we have
\begin{equation}
\der_\lambda\Phi\bigg|_{(z_0,r_0)}=-\der \exp_{x_0}\bigg|_{(t_0-r_0)E(z_0)}E(z_0):=-AE(z_0), \: 
\end{equation}
and
\begin{equation}
\der_z\Phi\bigg|_{(z_0,r_0)}=(t_0-r_0)A \der E\bigg|_{z_0}.
\end{equation}
Therefore for any $(v,\lambda)\in T_{z_0}\p M \times \R$ holds
\begin{equation}
\der\Phi\bigg|_{(z_0,r_0)}(v,\lambda)=A((t_0-r_0)\Theta v-\lambda E(z_0)).
\end{equation}
Since $A$ is invertible it holds that 
\begin{equation}
\der\Phi\bigg|_{(z_0,r_0)}(v,\lambda)=0 \quad \hbox{if and only if} \quad (t_0-r_0)\Theta v-\lambda E(z_0)=0.
\end{equation}
However this can only happen if $(v,\lambda)=0$ as $\Theta$ is injective and 
\begin{equation}
 \Theta v \in T_{E(z_0)}S_{x_0}M \quad \hbox{implies} \quad g_{E(z_0)}(E(z_0), \Theta v)=0.
\end{equation}
Due to the inverse function theorem $\Phi$ is a local diffeomoprhism near $(z_0,r_0)$.

\medskip
\textbf{(V)} Let $\Sigma$ be a $(n-1)$-dimensional surface in $M$ such that $\gamma$ is a curve on $\Sigma$ and $\gamma_{\nu}$ is transverse to $\Sigma$ at $r_0$. 
Due to implicit function theorem there exists a smooth function $\widehat r \colon U \to \R$ so that
\begin{equation}
\Phi(z,\lambda)\in \Sigma, \quad \hbox{if and only if} \quad \lambda=\widehat r(z) \hbox{ for all } z\in U, \quad \hbox{ and } \quad  \widehat r(z_0)=r_0.
\end{equation}
Since $r(\cdot)$ is continuous and satisfies 
$
\Phi(z,r(z))=\gamma(\rho(z))\in \Sigma
$
it holds that near $z_0$ functions $r$ and $\widehat r$ coincide. Therefore the map
\begin{equation}
\widetilde \Phi\colon U \ni z \mapsto \Phi(z,r(z)) \in \Sigma
\end{equation}
is smooth and its image is contained in the image of $\gamma$. However we have proven that $\widetilde \Phi$ is a local diffeomorphism  near $z_0$. Thus we arrive into a contradiction since $\widetilde \Phi$ maps $(n-1)$-dimensional surface $(n-1\geq 2)$ onto $1$-dimensional surface. Thus $r_0 < t_0$ is false. By an analogous argument we can prove that $r_0> t_0$ is also false and therefore it must hold that $r_0=t_0.$
\end{proof}

We are ready to present the proof for lemma \ref{lma:focusin_families_focus}.

\begin{proof}[Proof of lemma~\ref{lma:focusin_families_focus}]
The proof is an adaptation of \cite[proof of theorem 2.6, step 4]{kurylev2010rigidity} where one uses lemmas \ref{Le:con_of_total_time} and \ref{Le:intersect_of_geo_and_foc_fam}.
\end{proof}

\subsection{Boundary distance functions}
We start with giving a stronger formulation of lemma \ref{Le:intersect_of_geo_and_foc_fam}.

\begin{lemma}
\label{Le:intersect_of_geo_and_foc_fam_part2}
Let $z_0 \in \p M$, $ 0<t_0<\tau_{M}(z_0)$, and $F(z_0,t_0)=\{U,V(\cdot),t(\cdot)\}$ be a family of focusing directions around $(z_0,t_0)$. Let $\gamma$ be some geodesic that intersects every geodesic of $F(z_0,t_0)$ in the sense that 
\begin{equation}
\gamma(\rho(z))=\gamma_{V(z)}(r(z)),
\end{equation}
for some functions $\rho, r\colon U \to \R$ satisfying
\begin{equation}
0\leq r(z)\leq r_1<  \tau_{M}(z_0), \quad |\rho(z)|< L, \hbox{ for some } L>0.
\end{equation}
If in addition $h(z):=r(z)+\rho(z)$ is continuous then for any $z \in \p M$ which is close to $z_0$ holds that
\begin{equation}
\gamma(h(z_0)-t_0)=\gamma_{V(z)}(t(z)).
\end{equation}
That is all geodesics of $F(z_0,t_0)$ meet at the same point $\gamma(h(z_0)-t_0)$.
\end{lemma}
\begin{proof}
If $\gamma$ and $\gamma_{V(z_0)}$ are the same geodesic the result follows from lemma \ref{lma:focusin_families_focus}.

If we assume that $\gamma$ is a different geodesic to $\gamma_{V(z_0)}$ it the holds due to the definition of injectivity radius of $M$ that the geodesic segments $\gamma([-L,L])$ and $\gamma_{V(z_0)}([0,r_1])$ can intersect at most finitely many times. Thus there exists $N \in \N$ such that
$
\gamma(\rho)=\gamma_{V(z_0)}(r),$ for $(\rho,r) \in [-L,L]\times  (0,r_1)$ $ \hbox{if and only if } (\rho,r)=(\rho_0^k, r^k_0)$, for some $ \rho_0^k\in \{\rho_0^1,\ldots, \rho_0^N\} \subset [-L,L],$ and $r_0^k\in \{r^1_0,\ldots, r_0^N\} \subset (0,r_1).
$

Let $0< \epsilon < \frac{1}{2}\hbox{inj}(M)$. We claim that there exists $R_0>0$ such that for all $0<R<R_0$ and 
\begin{equation}
\label{eq:rhos_are_dense}
z \in U(R):=B(z_0,R) \cap \p M, \quad \hbox{ it holds that } \min_{j\in \{1,\ldots, N \}}|\rho(z)-\rho_0^j|<\epsilon.
\end{equation}
If this is not true, the boundedness of functions $r(\cdot)$ and $\rho(\cdot)$ imply the existence of a sequence $(z_k)_{k=1}^\infty \subset \p M$ that converges to $z_0$ and which satisfy
\begin{equation}
\lim_{k\to \infty} r(z_k)=\widetilde r \in [0,r_1] , \quad \lim_{k\to \infty} \rho(z_k)=\widehat \rho\in [-L,L]\setminus \{\rho_0^1,\ldots, \rho_0^N\}.
\end{equation}
Therefore we arrive in a contradiction  $\gamma(\widehat \rho)=\gamma_{V(z_0)}(\widehat r)$.

We set 
\vspace{-0.3cm}
\begin{equation}
\begin{split}
W_j(R)
:=&\{z\in U(R): \exists r\in [0,r_1], \: \exists \rho \in [-L,L], 
\\
&\hbox{ s.t.} \: \gamma_{V(z)}(r)=\gamma(\rho),\: r+\rho=h(z), \: |\rho-\rho_0^j| \leq \epsilon\}, \: j \in \{1,\ldots, N\},
\end{split}
\end{equation}
and show that these sets are relatively closed in $U(R)$. Choose a sequence $(z_k)_{k=1}^\infty \subset W_j(R)$ that converges to $z \in U(R)$. For any $k \in \N$ we choose $r_k \in [0,r_1]$ and $\rho_k \in [\rho_0^j-\epsilon, \rho_0^j+\epsilon]$ for which
\begin{equation}
\gamma_{V(z_k)}(r_k)=\gamma(\rho_k), \quad \hbox{ and } \quad r_k+\rho_k=h(z_k).
\end{equation}
After choosing a sub-sequence of $(z_k)_{k=1}^\infty$ we may without loss of generality assume that $\lim_{k\to \infty }r_k= \widetilde r \in [0,r_1]$, $\lim_{k\to \infty } \rho_k=\widetilde \rho \in [\rho_0^j-\epsilon, \rho_0^j+\epsilon]$,  
\begin{equation}
\gamma_{V(z)}(\widetilde r)=\gamma(\widetilde \rho), \quad \hbox{and} \quad h(z)=\widetilde r +\widetilde \rho.
\end{equation}
Here we used the continuity of the geodesic flow and the function $h(\cdot)$ to obtain the last two equations. Thus we have verified that $z \in W_j(R)$. Therefore $W_j(R)$ is closed and measurable. 

Equation \eqref{eq:rhos_are_dense} implies that
\begin{equation}
U(R)=\bigcup_{j=1}^N W_j(R).
\end{equation}
Thus there exists $j \in \{1,\ldots, N\}$ for which $W_j(R)$ has a strictly positive $(n-1)$-dimensional measure. Choose $k \in \{1,\ldots, N\}$ and suppose that $r_0^k\neq t_0$. By following the proof of lemma \ref{Le:intersect_of_geo_and_foc_fam} we show that after choosing a smaller $R>0$ there exist $0<T<\frac12\hbox{inj}(M)$, $(n-1)$-dimensional surface $\Sigma \subset M$ and a diffeomorphism 
\begin{equation}
\widetilde \Phi\colon U(R) \to \Sigma, \quad \hbox{ for which } \widetilde \Phi(z)\subset \gamma([\rho_0^k-T,\rho_0^k+T]), \hbox{ if } z \in W_k(R). 
\end{equation}
Thus the $(n-1)$-dimensional measure of $W_k(R)$ is zero. We have proven that the set $K \subset \{1,\ldots, N\}$ that contains all those $k \in \{1,\ldots, N\}$ for which $r_0^k=t_0$, is not empty and moreover
\begin{equation}
U(R)\setminus \left( \bigcup_{k \in K} W_k(R)\right), \hbox{ is of measure zero.}
\end{equation}
Therefore $\bigcup_{k \in K} W_k(R)$ is dense in $U(R)$. Since $\epsilon>0$ was arbitrary we can choose   $k \in K$, sequences $(z_j)_{j=1}^\infty \subset W_k(R)$, $(r_j)_{j=1}^\infty, \: (\rho_j)_{j=1}^\infty \subset \R$, such that
$
z_j \to z_0, \: \rho_j \to \rho_0^k, \: r_j \to \widetilde r \in [0,r_1]
$
and
$
\gamma(\rho_j)=\gamma_{V(z_j)}(r_j), \: \hbox{for every } j \in \N.
$
Since $r_0^k=t_0$ the continuity of geodesic flow implies
\begin{equation}
\gamma_{V(z_0)}(r_0^k)=\gamma(\rho_0^k)=\gamma_{V(z_0)}(\widetilde r).
\end{equation}
Thus we must have $t_0=\widetilde r$. Since the function $h(\cdot)$ is continuous we also have $\rho_0^k=h(z_0)-t_0$ and moreover
$
\gamma(h(z_0)-t_0)=\gamma_{V(z_0)}(t_0).
$
This and lemma \ref{lma:focusin_families_focus} complete the proof.
\end{proof}

\begin{lemma}
\label{lma:R-to-BCD}
The broken scattering relation determines the boundary cut distance function $\bou$.
\end{lemma}
\begin{proof}

Let $z_0 \in \p M.$ By the definition of $\bou(z_0)$ it holds that for $t_0 \in(0,\bou(z_0))$ the point $z_0$ is the closest boundary point to $x_0:=\gamma_{\nu(z_0)}(t_0)$. However if $t_0 > \bou(z_0)$ then there exists $w\in \p M$ that is closer to $x_0$ than $z_0$ in the sense that: 
\begin{equation}
\hbox{ There exists $s<t_0$ such that } \gamma_{\nu(w)}(s)=x_0.
\end{equation}
By lemma \ref{Le:comp_of_crit_func} we know that $\cut(z_0,\nu(z_0))>\bou(z_0)$. If
$t_0\in(\bou(z_0), \tau_M(z_0))$ then according to discussion at the beginning of subsection \ref{ssec:focus_fam} there exist $w \in \p M$ and a family of focusing directions $F(z_0,t_0):=\{U, V(\cdot),t(\cdot)\}$  around $(z_0,t_0)$ such that 
\begin{equation}
\label{eq:def-prop-of-BCD}
    \nu(w)R_{s+t(z)}V(z), \quad \hbox{and} \quad s<t(z) \quad \hbox{for all } z \in U.
\end{equation}

First we note that the broken scattering realtion determines the exit time function since for any $v \in \overline{\p_{in}SM}$ holds
\begin{equation}
\exit(v)=\sup\{t\geq 0: \hbox{ such that }vR_{2t}v\}.
\end{equation}
We claim that 
\begin{equation}
\begin{split}
\bou(z_0)=&\inf \{t_0 \in (0,\exit(\nu(z_0))): \hbox{ There exist $F(z_0,t_0)$, } \hbox{$s < t_0$,}
\\
&
\hbox{ and $w\in \p M\setminus\{z_0\}$ such that} \hbox{ \eqref{eq:def-prop-of-BCD} is valid}\}. 
\end{split}
\end{equation}
For the proof of this equation see in \cite[lemma 2.10]{kurylev2010rigidity}.
%
%
%
\end{proof}

We recall that the exponential map at the boundary is 
\[
\exp_{\p M}\colon \p M \times [0,\infty) \to M, \quad \exp_{\p M}(z,t)=\gamma_{\nu(z)}(t).
\]

\begin{proposition}
\label{pro:R-to-BDF}
Let $x\in M$. The broken scattering relation determines the set of boundary distance functions
\begin{equation}
r_x\colon \p M \to \R, \quad r_x(z)=d(x,z), \: z \in \p M, \: x \in M. 
\end{equation}
\end{proposition}
\begin{proof}
By the definition of boundary cut distance function we have that 
\begin{equation}
\begin{split}
M=\{\exp_{\p M}(z,t) \in M: z \in \p M, \: t \in [0, \bou(z)]  \}.
\end{split}
\end{equation}
Therefore $\exp^{-1}_{\p M}M \subset (\p M \times \R)$ is a way to represent $M$, although points in the boundary cut locus 
\begin{equation}
\omega_{\p M}:=\{x\in M: x=\exp_{\p M}(z,\bou(z)), \: z \in \p M\}
\end{equation}
may have several representatives. As we are interested in determining the boundary distance functions the possible ambiguity does not concern us since it is proven in \cite[proposition 3.1]{de2019inverse} that 
\begin{equation}
r_x=r_y \quad \hbox{ if and only if } \quad x=y \in M. 
\end{equation}

Let us fix $z_0 \in \p M$ and $t_0 \in [0,\bou(z_0)]$. Then we choose $w \in \p M$. As $\p M$ is strictly convex any distance minimizing curve from $w$ to $x:=\exp_{\p M}(z_0,t_0)$ is a geodesic $\gamma_{w,\eta}:[0,s]\to M$, where $s=d(w,x)$ and $\eta \in S_{w} M$. Due to lemma $\ref{Le:comp_of_crit_func}$ there exist a family of focusing directions $F(z_0,t_0)=\{U, V(\cdot),t(\cdot)\}$ such that
$
V(z)R_{s+t(z)}\eta, \: \hbox{ for any } z \in U.
$
We set 
\begin{equation}
S:=\{s>0: \hbox{ There exist $\eta \in S_wM$, and $F(z_0,t_0)$ s.t. $V(z)R_{s+t(z)}\eta$ holds}\}
\end{equation}
and claim that 
\begin{equation}
d(x,w)=\inf S.
\end{equation}
The proof of this claim is an adaptation of the proof of an analogous statement in \cite[theorem 2.13.]{kurylev2010rigidity}.  
%
%
\end{proof}

In the following lemma we give an invariant definition for families of focusing directions using the invariant definition of broken scattering relation given in theorem \ref{thm:main}. The claim of the lemma is a direct implication of the definition \ref{def:focusing_family}.

\begin{lemma}
\label{lma:focusing-families-conicide}
Let $(M_i,F_i)$ be two compact Finsler manifolds that satisfy the conditions of theorem \ref{thm:main}. Let $z_0 \in \p M_1$ and $t_0\in (0,\exit(\nu(z_0)))$. Then 
$
\mathcal{F}(z_0,t_0):=\{U,V(\cdot),t(\cdot)\}
$ 
is a family of focusing directions around $(z_0,t_0)$ if and only if $\widetilde{\mathcal{F}}(z_0,t_0):=\{\widetilde U,\widetilde V(\cdot),\widetilde t(\cdot)\}$ is a family of focusing directions around $(\diffeo(z_0),t_0)$, where
\begin{equation}
\widetilde U:=\diffeo (U)\subset \p M_2, \quad \widetilde V:=\Diffeo \circ V \circ \diffeo^{-1}, \quad  \hbox{ and } \quad \widetilde t:=t\circ \diffeo^{-1}.
\end{equation}
\end{lemma}

We are ready to prove proposition \ref{prop:R-to-bdf}.

\begin{proof}[Proof of proposition~\ref{prop:R-to-bdf}]

Let $x \in \hbox{int}M_1$. Choose $z_0 \in \p M_1$ and
\\
 $t_0 \in (0,\tau_{\p M_1}(z)]$ such that $x=\exp_{\p M_1}(z_0,t_0)$. Due to lemmas \ref{lma:R-to-BCD} and \ref{lma:focusing-families-conicide} it holds that $\tau_{\p M_2}(\diffeo (z))=\tau_{\p M_1}(z)$ and therefore
$
\tilde x: =\exp_{\p M_2}(\diffeo(z_0),t_0) \in \hbox{int}M_2.
$
We aim to verify  
\begin{equation}
\label{eq:distances-agree}    
d_1(z,x)=d_2(\diffeo(z),\widetilde x), \quad \hbox{ for all } z \in \p M_1.
\end{equation}
This implies the left hand side inclusion in \eqref{eq:BDF-agree}. By reversing the roles of $M_1$ and $M_2$, the analogous argument verifies the right hand side inclusion in \eqref{eq:BDF-agree}. Choose $z \in \p M_1$. We denote
$
s_1:=d_1(x,z), \hbox{ and } s_2:=d_2(\widetilde x, \diffeo (z)).
$
Since $\p M_1$ is strictly convex there exists a distance minimizing geodesics  $\gamma_{z,\eta}$ of $M_1$ from $z$ to $x$. 
Choose a family of focusing directions $\mathcal{F}(z_0,t_0)=\{U,V(\cdot),t(\cdot)\}$ such that
$
V(z)R^{(1)}_{s_1+t(z)}\eta, \: \hbox{for every } z \in U.
$
Then lemma \ref{lma:focusing-families-conicide} implies that $\widetilde{\mathcal{F}}(z_0,t_0)=\{\widetilde U, \widetilde V(\cdot),\widetilde t(\cdot)\}$ is a focusing family at $(\diffeo(z_0),t_0)$ that satisfy
$
\widetilde V(w)R^{(2)}_{s_1+\widetilde t(w)}\Diffeo \eta, \: \hbox{for every } w \in \widetilde U.
$
The proof of proposition \ref{pro:R-to-BDF} implies $s_2\leq s_1$. After reversing the roles of $x$ and $\widetilde x$, we use the analogous argument to verify $s_1\leq s_2$. We have proven $s_1=s_2$. Since $z \in M_1$ was arbitrary the equation \eqref{eq:distances-agree} is valid.
\end{proof}

\appendix
\section{Fourier integral operators and annihilators}

We consider $v, w \in \partial S M$ and introduce their duals, $\widetilde{v}, \widetilde{w} \in \partial (S^\ast M)$, via the Legendre transform \cite[formula (3.4)]{shen2001lectures}. Their tangential components in $T^* \partial M$ are written as $\widetilde{v}_{\partial}, \widetilde{w}_{\partial}$. Instead of the geodesic flow, $\phi_t$, we consider the co-geodesic or Hamiltonian flow, $\widetilde{\phi}_t$ on the co-sphere bundle and the canonical projection $\widetilde \pi\colon S^\ast M \to M$. We note that the broken scattering relation, given in definition~\ref{def:R}, is closely related to the following canonical relation
\begin{multline*}
   \Lambda := \left\{ \left(\widetilde{\pi}(\widetilde{v}),
       \widetilde{\pi}(\widetilde{w}), t,
       \widetilde{v}_{\partial},
       \widetilde{w}_{\partial},
       \tau \right) ;
       \left(x, \xi\right)\ \Big|\
       \widetilde{\pi}(\widetilde{\phi}_{t_1}(\widetilde{v}))
       = \widetilde{\pi}(\widetilde{\phi}_{t_2}(\widetilde{w}))
       = x ,\
       \right. \\ \left.
       t = t_1 + t_2 ,\
       \tau = F^\ast(\widetilde{\phi}_{t_1}(\widetilde{v}))
            = F^\ast(\widetilde{\phi}_{t_2}(\widetilde{w}))=1 ,\
   \xi = \widetilde{\phi}_{t_1}(\widetilde{v}) +
                       \widetilde{\phi}_{t_2}(\widetilde{w})
             \right\} .
\end{multline*}
We write $Y = \partial M \times \partial M \times (0,T)$ and $X =
M$. Then $\Lambda \subset T^*Y \times T^*X $. If the Finsler metric on $\p TM$ is known, as we assume in theorem \ref{thm:main}, the vectors
$\widetilde{v}_{\partial}, \widetilde{w}_{\partial}$ determine
$\widetilde{v}, \widetilde{w}$, respectively. The canonical relation
connects an element $(x,\xi)$ in the wavefront set of scatterers in $T^*X$ to an element in the wavefront set of scattered waves generated by sources, and detected by receivers, in $\partial M$ in
$T^*Y$. That is, there is a Fourier integral operator
(FIO) that maps scatterers to scattered waves restricted to $\partial
M$ and propagates singularities according to the mentioned canonical
relation $\Lambda$.

The Bolker condition states that the natural projection from $\Lambda$
to $T^*Y$ is injective. If the Bolker condition is
satisfied, the range of the FIO can be characterized by
pseudodifferential annihilators \cite{DeHoop-Uhlmann, Guillemin},
which yields the canonical relation. Thus, within the framework of the
results of this paper, vanishing annihilators determine the underlying
reversible Finsler manifold, if  both Bolker condition and the foliation condition of definition \ref{def:foliation} hold.

Although the Bolker condition should be fairly easy to satisfy, it
does not follow from the foliation condition. To see this, we consider
the Euclidean unit disc $M$ with a radial metric $c(r)e, \: r=|x|$. This manifold
can satisfy the Herglotz condition, 
$
\frac{\der}{\der r}\left(\frac{r}{c(r)}\right)>0,
$
and have a geodesic whose opening
angle is between $\pi$ and $2\pi$. Informally, this corresponds to a
geodesic going around the center. Such a geodesic will meet its
reflection across the center two times at some points $p,q \in \sisus{M}$. If we remove a segment of the
geodesic between the second intersection point $q$ and the nearest endpoint at the
boundary and the mirror image of this segment, we obtain two broken
rays $c_1$ and $c_2$ with exactly the same total length $t \in \R$ and boundary
data $v,w \in \p_{in}SM$. This situation violates the Bolker condition. We illustrate this setup in the following picture.

\begin{figure}[h]
\begin{picture}(220,220)
\label{Fi:f_p}
  \put(0,0){\includegraphics[width=8cm]{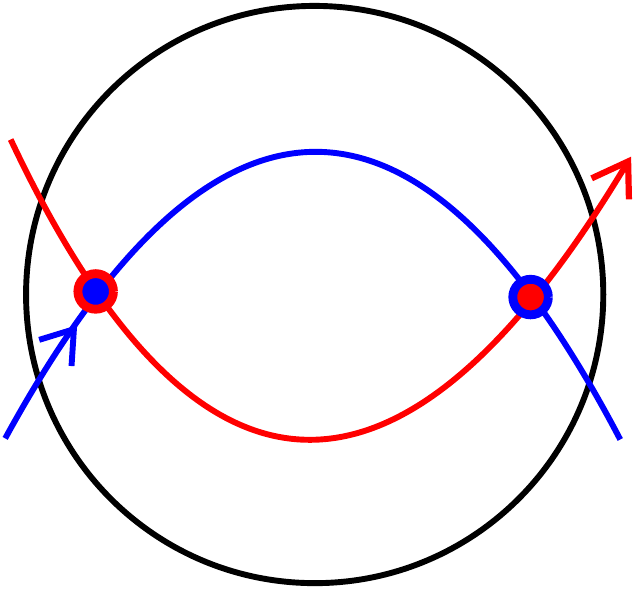}}
  \put(25,72){$v$}
   \put(172,103){$q$}
    \put(48,105){$p$}
    \put(240,150){$-w$}
   \end{picture}
   \caption{Here $c_1$ is the broken ray starting with initial velocity $v$, splitting at $p$ and exiting with velocity $-w$,  $c_2$ is the broken ray with same boundary conditions $\{v,-w\}$ that splits at $q$.}
\end{figure}
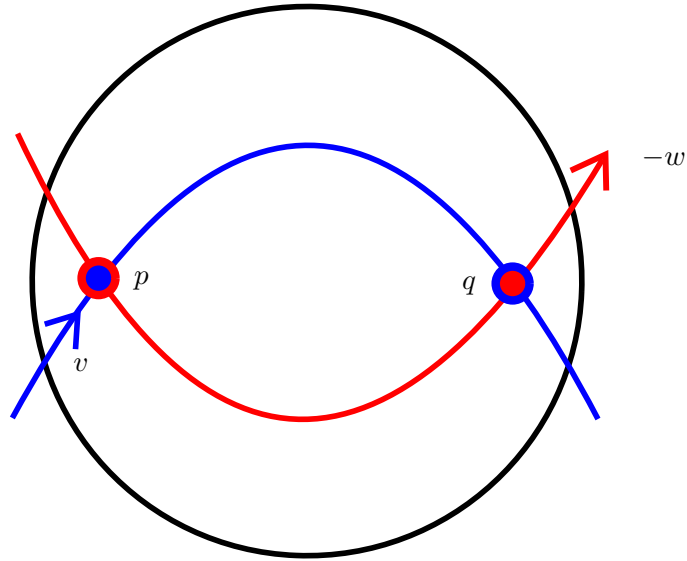

%


\subsection*{Acknowledgements}

MVdH was supported by the Simons Foundation under the MATH + X program, the National Science Foundation under grant DMS-1815143, and the corporate members of the Geo-Mathematical Imaging Group at Rice University USA.
JI was supported by the Academy of Finland (project 295853). 
ML was supported by Academy of Finland (projects 284715 and 303754).
TS was supported by the Simons Foundation under the MATH + X program and the corporate members of the Geo-Mathematical Imaging Group at Rice University.  

Part of this work was carried out during JI's and TS's visit to University of Washington, USA, and they are grateful for Prof. Gunther Uhlmann for hospitality and support. 

\bibliographystyle{abbrv}
\bibliography{bibliography}

\end{document}